\documentclass[preprint]{jsfds}
\usepackage[latin1]{inputenc}
\usepackage[T1]{fontenc}

\usepackage{natbib}

\startlocaldefs

\usepackage{url}
\RequirePackage{dsfont}

\newcommand{\R}{\mathbb{R}}

\newcommand{\egaldef}{:=} 
\newcommand{\defegal}{=:} 

\newcommand{\parenj}[1]{\mathopen{}\left( #1  \right) \mathclose{}}
\newcommand{\parens}[1]{( #1 )}
\newcommand{\parenb}[1]{\bigl( #1 \bigr)}

\newcommand{\croch}[1]{\left[ #1 \right]}
\newcommand{\crochj}[1]{\mathopen{}\left[ #1 \right] \mathclose{}}

\newcommand{\setj}[1]{\mathopen{}\left\{ #1 \right\} \mathclose{}}
\newcommand{\sets}[1]{\{ #1 \}}
\newcommand{\setb}[1]{\bigl\{ #1 \bigr\}}

\newcommand{\absb}[1]{\bigl\lvert #1 \bigr\rvert}

\newcommand{\norm}[1]{\left \lVert #1 \right\rVert}

\newcommand{\norms}[1]{\lVert #1 \rVert}
\newcommand{\normb}[1]{\bigl \lVert #1 \bigr\rVert}

\newcommand{\E}{\mathbb{E}} 
\DeclareMathOperator{\tr}{tr}
\DeclareMathOperator{\card}{card}

\DeclareMathOperator{\tmpargmin}{argmin}
\newcommand{\argmin}{\mathop{\tmpargmin}}
\DeclareMathOperator{\tmpargmax}{argmax}
\newcommand{\argmax}{\mathop{\tmpargmax}}
\newcommand{\Id}{I}

\newcommand{\Fh}{\widehat{F}}

\newcommand{\M}{\mathcal{M}}
\newcommand{\mM}{m \in \M}
\newcommand{\C}{\mathcal{C}}

\newcommand{\Fhm}{\Fh_m}
\newcommand{\sh}{\widehat{s}}

\newcommand{\shm}{\sh_m}

\newcommand{\mh}{\widehat{m}}

\newcommand{\Ch}{\widehat{C}}
\newcommand{\Chjumpgal}{\Ch_{\mathrm{jump}}}
\newcommand{\Chwindow}{\Ch_{\mathrm{window}}}

\newcommand{\Chslope}{\Ch_{\mathrm{slope}}}

\DeclareMathOperator{\pen}{pen}
\DeclareMathOperator{\crit}{crit}

\newcommand{\penmin}{\pen_{\min}} 
\newcommand{\opt}{\mathrm{opt}}
\newcommand{\penopt}{\pen_{\opt}}

\newcommand{\penoptz}{\pen_{\opt,0}}

\newcommand{\etamoins}{\eta_n^{-}}
\newcommand{\etaplus}{\eta_n^{+}}

\newcommand{\biaismax}{\mathcal{B}}

\newtheorem{theorem}{Theorem}

\RequirePackage{amssymb}
\renewcommand{\leq}{\leqslant}
\renewcommand{\geq}{\geqslant}

\newcommand{\refsurvey}[1]{\textbf{#1}}

\renewcommand{\thesection}{R.\arabic{section}}

\endlocaldefs

\setmainlanguageenglish
\firstpage{1} 

\begin{document}

\begin{frontmatter}

  \title{Rejoinder on: Minimal penalties and the slope heuristics: a survey}
  \runtitle{Rejoinder on: Minimal penalties and the slope heuristics: a survey}

  \begin{aug}
    \auteur{%
      \prenom{Sylvain} \nom{Arlot}%
      \thanksref{t1}%
      \contact[label=e1]{sylvain.arlot@u-psud.fr}%
    }%
\affiliation[t1]{%
Universit\'e Paris-Saclay, Univ.\@ Paris-Sud, CNRS, Inria, Laboratoire de math\'ematiques d'Orsay, 91405, Orsay, France.
\printcontact{e1}
}
    \runauthor{S. Arlot}
  \end{aug}

\end{frontmatter}
  
\setcounter{footnote}{0} 


I am grateful to the Editor-in-Chief for arranging such an informative and stimulating discussion of my survey paper \citep{Arl:2019:surveypenmin:journal}, 
and to all discussants for their inspiring comments. 
The discussion nicely complements the paper on various aspects. 
Let me add a few comments on these, following approximately the order of the paper. 

For the sake of clarity, references to the survey paper are made with bolded numbers.

\section{Slope heuristics when there is some model bias (Sections \refsurvey{2} and~\refsurvey{8.2})}

Gilles Celeux and Christine Keribin ask whether it is possible to 
justify the slope heuristics when \emph{all} models are biased, 
assuming that the bias is constant over the large models. 
The answer is positive. 
For instance, modifying slightly the proof of Theorem~\refsurvey{1} in Section~\refsurvey{2}, 
we obtain the following result. 
\begin{theorem} \label{thm.OLS.with-bias}
In the framework described in Section~\refsurvey{2.1} of the survey, 
assume that $\M$ is finite and that  
\begin{equation}
\label{hyp.thm.OLS.Gauss}
\tag{\ensuremath{\mathbf{HG}}}
\varepsilon \sim \mathcal{N}(0,\sigma^2 \Id_n)
\, . 
\end{equation}
Let $m_1 \in \M$ be such that 
\begin{equation} 
\label{eq.thm.OLS.with-bias.hyp}
B \egaldef \min_{m \in \M} \setj{ \frac{1}{n} \normb{ (\Id_n - \Pi_{m}) F}^2 } = \frac{1}{n} \normb{ (\Id_n - \Pi_{m_1}) F}^2 
\end{equation}
and assume that $m_2 \in \M$ exists such that $D_{m_2} \leq D_{m_1} / 20$. 
Recall that for every $C \geq 0$, $\mh(C)$ is defined by Eq.~\refsurvey{(10)} of the survey. 
Then, for every $\gamma \geq 0$, some $n_0(\gamma)$ exists such that if $n \geq n_0(\gamma)$, 
with probability at least $1 - 4 \card(\M) n^{-\gamma}$, 
the following inequalities hold simultaneously\textup{:}
\begin{align}
\label{eq.thm.OLS.with-bias.Cpt-Dgrd}
\forall C \leq \parenb{ 1 - \etamoins } \sigma^2
\, ,
\qquad 
D_{\mh(C)} &\geq \frac{9 D_{m_1}}{10} 
\, , 
\\
\label{eq.thm.OLS.with-bias.Cgrd-Dpt}
\text{and} \qquad 
\forall C \geq \parenb{ 1 + \etaplus } \sigma^2
\, ,
\qquad 
D_{\mh(C)} &\leq \frac{D_{m_1}}{10}
\, , 
\end{align} 
where 
\begin{align*}
\sigma^2 \etamoins 
&= \frac{n}{D_{m_1}}  \parenb{ 
57 
\sigma \sqrt{B}  
+ 41 \sigma^2 } \sqrt{\frac{\gamma \log(n)}{n}} 
\, , 
\\ 
\sigma^2 \etaplus 
&= \frac{n}{D_{m_1}} \crochj{ 
20 ( B' - B ) 
+ \parenb{ 114 
\sigma \sqrt{ B' } 
+ 82 
\sigma^2 }
\sqrt{\frac{\gamma \log(n)}{n}} 
}
\, , 
\\ 
\text{and} \qquad 
B' 
&\egaldef  \biaismax \parenj{ \frac{D_{m_1}}{20}} 
= \inf_{\mM \,/\, D_m \leq D_{m_1}/20} \setj{\frac{1}{n} \normb{ (\Id_n - \Pi_m) F}^2 } 
\, .
\end{align*}
On the same event, Eq. \refsurvey{(17)} and~\refsurvey{(18)} of the survey also hold true. 
\end{theorem}
\begin{proof}
We adapt the proof of Theorem~\refsurvey{1} of the survey. 
We keep unchanged the definition of $n_0(\gamma)$, step~1 and step~2.1 of the proof, 
as well as step~3 which still applies for proving Eq. \refsurvey{(17)}--\refsurvey{(18)}. 
Only steps 2.2 and 2.3 need to be modified, as detailed below. 
%
Since we use it repeatedly in the proof, 
let us recall Eq.~\refsurvey{(25)} of the survey, 
which holds true on the event $\Omega_x$ 
considered in the proof of Theorem~\refsurvey{1}: 
\begin{align}
\tag{\refsurvey{25}} \label{R.eq.GCM-critC}
\forall m\in \M \, , \qquad 
\absb{G_C(m) - \crit_C(m) }
&\leq 2 \sigma^2 \parenj{ \sqrt{ \frac{x}{n} } + \frac{x}{n} }  + \frac{2 \sigma \sqrt{2x}}{n} \normb{(\Id_n- \Pi_m) F}
\\ 
\tag{Eq.~(\refsurvey{11})}
\text{where} \qquad 
 \crit_C(m) &\egaldef \frac{1}{n} \crochj{ \normb{ (\Id_n - \Pi_m) F}^2 + ( C - \sigma^2) D_m }
\, .
\end{align}

\paragraph{Step 2.2': lower bound on $D_{\mh(C)}$ when $C$ is too small (proof of Eq.~\eqref{eq.thm.OLS.with-bias.Cpt-Dgrd})}
Let $C \in [0,\sigma^2)$. 
Since $\mh(C)$ minimizes $G_C(m)$ over $\mM$, it is sufficient to prove that if $C \leq \parens{ 1 - \etamoins } \sigma^2$, 
\begin{equation}
\label{eq.pr.thm.OLS.with-bias.Cpt-Dgrd.but}
G_C(m_1) < \inf_{\mM , \, D_m < 9 D_{m_1} / 10} \setb{G_C(m)} \, . 
\end{equation}
On the one hand, by Eq. \eqref{R.eq.GCM-critC} and~\eqref{eq.thm.OLS.with-bias.hyp}, 
\begin{align}
G_C(m_1)
&\leq 
\crit_C(m_1) + 2 \sigma^2 \parenj{ \sqrt{ \frac{x}{n} } + \frac{x}{n} } 
+ 2 \sigma \sqrt{ \frac{2 x}{n} } \sqrt{B}  
\notag 
\\ 
&= B + (C-\sigma^2) \frac{D_{m_1}}{n} + 2 \sigma^2 \parenj{ \sqrt{ \frac{x}{n} } + \frac{x}{n} } 
+ 2 \sigma \sqrt{ \frac{2 x}{n} } \sqrt{B}  
\label{eq.pr.thm.OLS.with-bias.Cpt-Dgrd.1}
\, .
\end{align}
On the other hand, by Eq.~\eqref{R.eq.GCM-critC}, for any $\mM$ such that $D_m < 9 D_{m_1} / 10$,
\begin{align}
G_C(m)
&\geq
\frac{(C-\sigma^2) D_m}{n}
- 2 \sigma^2 \parenj{ \sqrt{ \frac{x}{n} } + \frac{x}{n} }
+ \frac{1}{n} \normb{(\Id_n - \Pi_m) F}^2 - \frac{2\sigma \sqrt{2x}}{n} \normb{(\Id_n - \Pi_m) F}
\notag
\\
&> \frac{9 D_{m_1}}{10 n} (C-\sigma^2)
- 2 \sigma^2 \parenj{ \sqrt{ \frac{x}{n} } + \frac{x}{n} } 
+ B - 2\sigma \sqrt{\frac{2x}{n} } \sqrt{B}  
- \frac{2 \sigma^2 x}{n}
\, , 
\label{eq.pr.thm.OLS.with-bias.Cpt-Dgrd.2}
\end{align}
using that 
$u - 2 \alpha \sqrt{u} 
\geq B - 2 \alpha \sqrt{B} - \alpha^2$ 
when $u \geq B \geq 0$, 
with
\begin{equation*}
\alpha = \sigma \sqrt{\frac{2x}{n} }
\qquad \text{and} \qquad 
u = \frac{1}{n} \normb{(\Id_n - \Pi_m) F}^2 \geq  B
\quad 
\text{by Eq.~\eqref{eq.thm.OLS.with-bias.hyp}.} 
\end{equation*}
To conclude, the upper bound in Eq.~\eqref{eq.pr.thm.OLS.with-bias.Cpt-Dgrd.1} is smaller than the lower bound in Eq.~\eqref{eq.pr.thm.OLS.with-bias.Cpt-Dgrd.2} when
\begin{equation}
\label{eq.pr.thm.OLS.with-bias.Cpt-Dgrd.3}
C \leq \sigma^2 -  \frac{20 n}{D_{m_1}} \parenj{ 2 \sigma^2 \sqrt{\frac{x}{n}} + 3 \sigma^2 \frac{x}{n} + 2 \sigma \sqrt{\frac{2x}{n}} \sqrt{B} }
\defegal \widetilde{C}'_1(x) \, .
\end{equation}
Taking $x=\gamma \log(n)$, for $n \geq n_0(\gamma)$, 
we have $\widetilde{C}'_1(x) \geq \sigma^2 (1 - \etamoins)$ hence Eq.~\eqref{eq.thm.OLS.with-bias.Cpt-Dgrd}.

\paragraph{Step 2.3': upper bound on $D_{\mh(C)}$ when $C$ is large enough (proof of Eq.~\eqref{eq.thm.OLS.with-bias.Cgrd-Dpt})}
Let $C>\sigma^2$.
Similarly to the proof of Eq.~\eqref{eq.thm.OLS.with-bias.Cpt-Dgrd}, it is sufficient to prove that 
if $C \geq \parens{ 1 + \etaplus } \sigma^2$,
\begin{equation}
\label{eq.pr.thm.OLS.with-bias.Cpt-Dpt.but.alt2}
G_C(m_2) < \inf_{\mM \, , \, D_m > D_{m_1} / 10} \setb{ G_C(m) }
\end{equation}
where $m_2 \in \argmin_{\mM \,/\, D_m \leq D_{m_1}/20} \setb{ \norms{(\Id_n - \Pi_m) F}^2 }$ exists by assumption. 
On the one hand, by Eq.~\eqref{R.eq.GCM-critC}, 
\begin{align}
G_C(m_2)
&\leq \crit_C(m_2)
+ 2 \sigma^2 \parenj{ \sqrt{ \frac{x}{n} } + \frac{x}{n} } + \frac{2 \sigma \sqrt{2x}}{n} \normb{(\Id_n- \Pi_{m_2}) F}
\notag
\\
&\leq 
B' 
+ \frac{(C-\sigma^2) D_{m_1}}{20 n}
+ 2 \sigma^2 \parenj{ \sqrt{ \frac{x}{n} } + \frac{x}{n} }
+ 2 \sigma \sqrt{ \frac{2x}{n} } \sqrt{ B' }
\label{eq.pr.thm.OLS.with-bias.Cpt-Dpt.1.alt2}
\, .
\end{align}
On the other hand, as in step 2.2', 
using Eq.~\eqref{R.eq.GCM-critC} and Eq.~\eqref{eq.thm.OLS.with-bias.hyp}, 
we get that for any $\mM$ such that $D_m > D_{m_1}/10$,
\begin{align}
G_C(m)
&> \frac{D_{m_1}}{10 n} (C-\sigma^2) - 2 \sigma^2 \parenj{ \sqrt{ \frac{x}{n} } + \frac{x}{n} } 
+ B - 2 \sigma \sqrt{ \frac{2x}{n} } \sqrt{B} 
- \frac{2 \sigma^2 x}{n}
\label{eq.pr.thm.OLS.with-bias.Cpt-Dpt.2.alt2}
\, . 
\end{align}
To conclude, 
the upper bound in Eq.~\eqref{eq.pr.thm.OLS.with-bias.Cpt-Dpt.1.alt2}
is  smaller than the lower bound in Eq.~\eqref{eq.pr.thm.OLS.with-bias.Cpt-Dpt.2.alt2}
when
\begin{equation}
\label{eq.pr.thm.OLS.with-bias.Cpt-Dpt.3}
C  \geq \sigma^2 + \frac{20 n}{D_{m_1}} \crochj{ 
B' - B 
+ 2 \sigma \sqrt{ \frac{2x}{n} } \parenb{ \sqrt{B} + \sqrt{B'} } 
+ 2 \sigma^2 \parenj{ 2 \sqrt{\frac{x}{n}} + \frac{3 x}{n} } } 
\defegal \widetilde{C}'_2(x)
\, .
\end{equation}
Taking $x = \gamma \log(n)$, for $n \geq n_0(\gamma)$, 
we have $\widetilde{C}'_2(x) \leq \sigma^2 (1+\etaplus)$ 
since $B \leq B'$, 
hence Eq.~\eqref{eq.thm.OLS.with-bias.Cgrd-Dpt}. 
\end{proof}
Theorem~\ref{thm.OLS.with-bias} can also be generalized to sub-Gaussian noise, 
following Remark~\refsurvey{1} in Section~\refsurvey{2.5}. 
Note that an equivalent of Eq.~\refsurvey{(15)} ---with an additional term depending on $B$--- 
could certainly be obtained under the assumptions of Theorem~\ref{thm.OLS.with-bias}. 

Assume that $D_{m_1} > \frac{10}{17} \max_{m \in \M} D_m\,$; 
for instance, this assumption holds true when the models are nested 
since one can always take $m_1 \in \argmax_{m \in \M} D_m\,$. 
Then, 
\[ 
\frac{9 D_{m_1}}{ 10 } - \frac{ D_{m_1} }{10} > \max_{m \in \M} \sets{ D_m } - \frac{9 D_{m_1}}{ 10 } \, , 
\] 
so Theorem~\ref{thm.OLS.with-bias} shows that $\Chwindow(\eta)$ 
with $\eta = \max\sets{\etamoins,\etaplus}$ is close to $\sigma^2$; 
a formal statement can be obtained, 
similarly to Proposition~\refsurvey{3} in Section~\refsurvey{6.1}. 

Theorem~\ref{thm.OLS.with-bias} shows one possible generalization of Theorem~\refsurvey{1}, 
that allows all models to be biased ($B>0$). 
The particular case $\Pi_{m_1}=\Id_n$ ---hence $B=0$ and $D_{m_1}=n$--- yields 
Theorem~\refsurvey{1}, up to numerical constants\footnote{We can even recover the same constants as in Theorem~\refsurvey{1} by using Eq.~\eqref{eq.pr.thm.OLS.with-bias.Cpt-Dpt.3} 
in the proof of Theorem~\ref{thm.OLS.with-bias}.}. 
When $B>0$ and $D_{m_1}<n$, $\etamoins$ and $\etaplus$ are slightly larger 
than in Theorem~\refsurvey{1} ---hence the estimation of $\sigma^2$ by the slope heuristics 
could be less precise---, 
but this loss is at most by a multiplicative factor of order 
\[ 
\frac{n}{D_{m_1}} 
\max \setj{ 1 \, , \, \sqrt{\frac{B'}{\sigma^2} } } 
\, . 
\]
For instance, if $\sigma^2$ stays constant while 
$\M=\M_n\,$, $m_1 = m_{1,n}\,$, $B=B_n$ and $B'=B'_n$ are such that 
\[ 
D_{m_{1,n}} \gg \sqrt{ n \log(n) } \, , 
\qquad 
\frac{B'_n}{\sigma^2} \in \mathcal{O}(1) 
\qquad 
\text{and} \qquad 
B'_n - B_n  \ll \frac{\sigma^2 D_{m_{1,n}}}{n} 
\]
as $n$ tends to infinity, 
then 
\[ 
\lim_{n \to +\infty} \max\sets{ \etamoins,\etaplus } = 0 
\, , 
\] 
hence the slope heuristics yields a consistent estimation of~$\sigma^2$. 

Therefore, contrary to what is suggested at the end of Section~\refsurvey{5.2.2} and in Section~\refsurvey{8.2}, 
having a model with a small approximation error in $(S_m)_{m \in \M}$ 
is not an ``unavoidable assumption'' for the slope heuristics to work. 
Could it be relaxed even more\footnote{Following Gilles Celeux's footnote, having jumped onto the Moon, I now hope that some researchers will be able to explore Mars.} than what is done by Theorem~\ref{thm.OLS.with-bias} above?

\section{Related procedures (Section~\refsurvey{6})}
\label{R.sec.related}
\subsection{Elbow heuristics (Section~\refsurvey{6.4})}

Some \emph{elbow-heuristics} procedure is present in the simulation experiments of Servane Gey 
---under the name $\widehat{m}_{plateau}$---, 
for the calibration of CART in binary classification. 
I notice that it performs decently, but a ``jump'' formulation of the slope heuristics 
seems to outperform it, at least in the four settings considered. 

\subsection{Scree test and related methods (Section~\refsurvey{6.5})}
\label{R.sec.related.scree}
David Donoho and Matan Gavish provide an interesting insight into the \emph{scree test} 
for choosing the number of factors in factor analysis, in the light of random-matrix theory. 
Their assessment of several eigenvalue-thresholding procedures points out 
several important ideas and questions. 

Universal singular value thresholding \citep[USVT;][]{Cha:2015} is another procedure 
related to minimal-penalty algorithms, not mentioned in the survey. 
In the setting considered by David Donoho and Matan Gavish, 
USVT suggests to take a threshold $\lambda$ 
``just above'' the minimal threshold $\lambda_{\min} = 2 \sqrt{n} \sigma$, 
which is suboptimal since \citet{Gav_Don:2014} show that 
the optimal threshold is $(4/\sqrt{3}) \sqrt{n} \sigma = (2/\sqrt{3}) \lambda_{\min}\,$. 
As in the slope heuristics, there is here a factor $2/\sqrt{3} \approx 1.15 > 1$ between 
the minimal and optimal values of the tuning parameter 
(the constant $C$ in front of the penalty for 
the slope heuristics, the threshold $\lambda$ here). 
Therefore, I consider these results as perfectly consistent with what is known about 
minimal-penalty algorithms. 

A remarkable difference with the slope heuristics is that this factor is not equal to~$2$, 
and it is not universal: 
in the white-noise model, it depends on the aspect ratio $m/n$ of the matrix, 
and in general it depends on the noise distribution. 
%
%
From the practical point of view, 
as detailed in Section~\ref{R.sec.openpb.phase-transitions}, 
a universal relationship between the minimal threshold and 
a good (but suboptimal) threshold ---such as USVT--- would be useful 
when the noise distribution is unknown, 
unless a fully data-driven optimal procedure is available. 

Note finally that Sylvain Sardy points out another good competitor of the scree test: 
quantile universal thresholding \citep[QUT;][]{Gia_etal:2017}.

\section{Practical aspects of minimal-penalty algorithms (Section~\refsurvey{7})}
\label{R.sec.practice}

\'Emilie Devijver judiciously underlines the need for theoretical analyses of 
minimal-penalty algorithms \emph{exactly as they are used}, 
which can differ from the general-purpose 
algorithms defined in Sections~\refsurvey{2}--\refsurvey{3}. 
This raises several theoretical open problems that are missing in the survey, 
such as theoretical grounds for the nested minimal-penalty algorithm (Section~\refsurvey{7.3}), 
or for choosing a penalty shape (Sections~\refsurvey{7.4} and \refsurvey{7.6}; see also Section~\ref{R.sec.practice.choice} below).

\subsection{Definitions of \texorpdfstring{$\widehat{C}$}{Chat} (Section~\refsurvey{7.1})} 
\label{R.sec.practice.def}

The discussions emphasize that choosing between $\Chjumpgal$ and $\Chslope$ is not an easy task: 
Gilles Celeux argues for $\Chslope$ ---thinking especially of a collection of nested models---, 
while \'Emilie Lebarbier argues for $\Chjumpgal$ for change-point detection 
---given that both $\Chjumpgal$ and $\Chslope$ have drawbacks in that framework, 
but the ones of $\Chjumpgal$ are easier to handle---, 
and \'Emilie Devijver suggests to visualize both $\Chjumpgal$ and $\Chslope$ when $(S_m)_{m \in \M}$ 
is a union of subcollections! 
Overall, this confirms that one should not use a single definition of $\Ch$ blindly, 
especially when using it for the first time in a given framework.

Servane Gey uses a variant of $\Chjumpgal$ 
---called $\mh_{jump}$ in her contribution--- 
which is not presented in the survey: 
it is a nice alternative to $\Chwindow$ for merging very close jumps.

Finally, in frameworks where both the jump and slope approaches have important drawbacks 
---for instance, change-point detection---, 
other representations of the empirical risk could be useful for estimating the minimal penalty. 
I mention this as an open problem, 
for which some inspiration might come from the suggest of David Donoho and Matan Gavish 
to use a histogram instead of a scree plot. 
In their setting, this suggest comes from modern random-matrix theory; 
for minimal-penalty estimation, existing and future probabilistic results might suggest other meaningful representations 
of the empirical risk.

\subsection{Choice of a penalty shape (Sections~\refsurvey{7.4} and \refsurvey{7.6})} 
\label{R.sec.practice.choice}
The question of choosing a penalty shape, for instance between a simple and a complex one, 
is raised by several discussants. 
Bertrand Michel provides several examples where theory provides ``complex'' penalty shapes 
---that is, depending on several unknown constants---, which makes their data-driven calibration 
difficult. 
Then, one may consider instead a ``simplified'' shape ---depending only on one unknown constant---, 
but it is often not obvious that such a simplification improves the performance 
compared to the multivariate extension of the slope approach. 
Section~\refsurvey{7.4} reports numerical experiments illustrating the difficulty of this choice, 
and the numerical experiments of \'Emilie Lebarbier for change-point detection 
provide another example where both options face practical issues. 

Then, two strategies are proposed by the discussants. 
On the one hand, \'Emilie Devijver shows that a precise enough mathematical study 
---with upper \emph{and} lower risk bounds--- 
can help answering the question. 
On the other hand, Bertrand Michel asks for a data-driven procedure choosing among several penalty shapes 
---which is another (elaborate) estimator selection problem---, \emph{with theoretical grounds}. 
Both options raise challenging mathematical open problems, 
and I think that solving any of them would have a significant practical impact. 

The second strategy could be solved thanks to minimal-penalty algorithms, 
because, as written by Bertrand Michel, ``it is very easy to detect when the method goes wrong''. 
This suggests an easy way of discarding a candidate penalty shape\footnote{We implicitly assume here that 
the slope heuristics is valid, or at least that the minimal and optimal penalties are proportional. 
Otherwise, as in the setting of Section~\refsurvey{3.3}, a data-driven choice of the 
optimal-penalty shape seems to be too difficult a problem. }. 
Nevertheless, Jean-Patrick Baudry shows that \emph{validating} a penalty shape is a more difficult issue: 
according to some ongoing research about model-based clustering, 
the slope heuristics can fail ---by underfitting--- even when the empirical risk 
has a clear linear behaviour! 
This can be an issue about which practicioners using minimal-penalty algorithms 
should be warned. 
Fortunately, minimal-penalty algorithms cannot strongly overfit by construction, 
so this issue should only have a moderate impact on the risk, except in pathological settings.

\section{Conjectures and open problems on minimal-penalty algorithms (Sections~\refsurvey{8.3} and \refsurvey{8.5})}
\label{R.sec.openpb}

\subsection{Need for theory with various risk functions} 
\label{R.sec.openpb.need-theory}

Yannick Baraud clearly shows that the least-squares risk 
---which is considered by most theoretical results about minimal-penalty algorithms--- 
can be a bad choice, except in the Gaussian framework; 
Lucien Birg\'e also argues in this direction. 
This points out the strong need for theoretical results on minimal-penalty algorithms (and beyond) 
with various risk functions. 
Let me explain why. 

Following \citet{Bre:2001:2cultures}, two main points of view exist in statistical modeling. 
On the one hand, one can assume a stochastic data model and aim at estimating its parameters, in order to get an 
estimation of the full data-generation mechanism. 
With such a goal, the performance measure of any procedure should reflect the quality of this estimation. 
As a consequence, as explained by Yannick Baraud and Lucien Birg\'e, 
the Hellinger or total-variation distances between probability distributions should be used, 
while the least-squares risk can be misleading for non-Gaussian data in the regression framework. 

On the other hand, one can consider the data-generation mechanism as a black box 
and stick to predictive accuracy (with respect to some application-dependent risk function) 
for evaluating the performance of a statistical procedure. 
This can be seen as the ``machine-learning'' point of view, as opposed to the previous one which could be called ``statistical''. 

In both cases, the key observation is that one should use a performance measure 
---that is, a risk function, in the setting of Section~\refsurvey{3.1}--- 
that is \emph{adapted to the final goal}. 
If one wants to model precisely the full data-generation mechanism, 
the risk should be a distance between probability distributions (Hellinger or total variation, for instance). 
If the final goal is to build an efficient predictor, 
one should use a prediction risk with an appropriate cost function 
(such as the quadratic, absolute value, hinge, logistic or exponential loss in binary classification). 
And other goals are possible, depending on the task to be solved 
(for instance, estimating a specific part of the data-generation mechanism, as in semi-parametric statistics). 

On the contrary, choosing a risk function because it is the only one for which theoretical results are available 
is not satisfactory at all. 
So, there must be a strong effort by theoreticians to provide mathematical guarantees for various performance measures, 
adapted to practical needs. 
In particular, in addition to the open problems already mentioned in the survey, 
an important open problem would be to provide minimal-penalty algorithms 
with theoretical guarantees in terms of Hellinger or total-variation distance. 
Lucien Birg\'e points out a more specific interesting open problem 
that is related to the above question: 
could a minimal-penalty algorithm 
(or an observable phase transition, as suggested in Section~\refsurvey{8.5}) 
be used for the fine tuning of the penalty appearing inside $\rho$-estimators 
\citep{Bar_Bir_Sar:2017,Bar_Bir:2018}? 

Up to now, the most advanced results about minimal-penalty algorithms with general risk functions 
are the ones of \citet[Chapters~7--8]{Sau:2010:phd} 
on minimum-contrast estimators with a regular contrast and the associated risk. 
In the discussion, Adrien Saumard proposes an approach for 0--1 classification that might work more generally for bounded M-estimation, 
hence leading to further improvements in this direction. 
Nevertheless, obtaining results for the Hellinger and total-variation distances requires new ideas, 
since these risk functions are not the expectation of a contrast. 

\subsection{Time-series analysis} 
I would like to thank Adrien Saumard for pointing out an important open problem, 
which is to develop minimal-penalty algorithms for time-series analysis, 
for which cross-validation and resampling procedures can be difficult to use for estimator selection. 
A subsection on this topic was indeed missing in Section~\refsurvey{8.3}. 

In addition to the references provided by Adrien Saumard 
and a few results mentioned in the survey that deal with dependent data 
\citep{Ler:2010:mixing,Gar_Ler:2011}, 
a framework in which theoretical results can be obtained with dependent data 
is the one of Section~\refsurvey{2} with assumption \eqref{hyp.thm.OLS.Gauss} 
replaced by $\varepsilon \sim \mathcal{N}(0,\Sigma)$ 
for some general covariance matrix~$\Sigma$. 
We can always write $\Sigma = A^{\top} A$ with $A \in \mathcal{M}_n(\R)$ symmetric, 
so that $\varepsilon = A \xi$ with $\xi \sim \mathcal{N}(0,\Id_n)$. 
Then, by Eq.~\refsurvey{(5)}--\refsurvey{(6)}, 
we obtain the expectations of the risk and empirical risk: 
\begin{align}
\notag 
\E\croch{\frac{1}{n} \norm{ \Fhm - F }^2} 
&= \frac{1}{n} \normb{ (\Id_n - \Pi_m) F}^2 + \frac{\tr( \Sigma \Pi_m )}{n}  \, , 
\\
\notag 
\text{and} \qquad 
\E\croch{\frac{1}{n} \norm{ \Fhm - Y }^2} 
&= \frac{1}{n} \normb{ (\Id_n - \Pi_m) F}^2 + \frac{\tr( \Sigma )}{n}  - \frac{\tr( \Sigma \Pi_m )}{n}  \, , 
\end{align}
so that 
\[ 
\penopt^{\Sigma} (m) 
:= \penoptz^{\Sigma} (m) + \frac{\tr( \Sigma )}{n}
= \frac{ 2 \tr( \Sigma \Pi_m )} {n} 
\] 
is an optimal penalty, while 
\[ 
\penmin^{\Sigma}(m) = \frac{ \tr( \Sigma \Pi_m )} {n} 
\] 
is a minimal penalty. 
A remarkable consequence is that 
\[ 
\penopt^{\Sigma} (m) = 2 \penmin^{\Sigma}(m) 
\] 
holds true whatever the noise dependence structure! 
Furthermore, the proof of Theorem~\refsurvey{1} can certainly be extended to this setting, 
with 
\[ 
\mh(C) \in \argmin_{\mM} \setj{ \frac{1}{n} \norm{\Fhm - Y}^2 + C \frac{\tr( \Sigma \Pi_m )}{n} } 
\, , 
\] 
$D_m$ replaced by $\C_m = \tr(\Sigma \Pi_m)$, 
a phase transition at $C=1$, 
and an optimal oracle inequality for $C=2$. 
The key argument is to get an equivalent of Eq. \refsurvey{(23)}--\refsurvey{(24)} 
in the dependent case, which can follow from standard Gaussian concentration results 
\citep[Propositions 4 and~6]{Arl_Bac:2009:minikernel_long_v2} under appropriate assumptions on $\Sigma$ and $(\Pi_m)_{\mM}\,$. 
We leave this problem as an exercise to the reader. 
It seems reasonable to conjecture that the results of 
\citet{Arl_Bac:2009:minikernel_long_v2} on linear estimators can also be extended 
to dependent data, under appropriate conditions on $\Sigma$ and the matrices $(A_m)_{\mM}\,$. 

However, the above result provides theoretical grounds to a slope-heuristics algorithm 
that can be used \emph{only if} $\tr(\Sigma \Pi_m)$ is known 
(up to a multiplicative constant) for all $\mM$. 
Such an assumption is strong, but it might be satisfied in some specific cases, 
either exactly ---for instance for linear inverse problems with independent homoscedastic noise--- 
or approximately ---with assumptions on the dependence structure. 

Let us finally remark that the above arguments can also serve as a starting point for analyzing 
the heteroscedatic-regression setting mentioned in point~1 of the first 
paragraph of Section~\refsurvey{8.3.7}. 

\subsection{Supervised classification (Section~\refsurvey{8.3.1})} 

Servane Gey's contribution complements Section~\refsurvey{8.3.1} 
(and the end of Section~\refsurvey{8.3.7}) well, 
with some numerical experiments about the slope heuristics 
for tuning CART in binary classification. 
This should increase the motivation for working on this open problem, 
and Adrien Saumard actually provides us a well-motivated conjecture 
on this very topic.

\subsection{Model-based clustering (Section~\refsurvey{8.3.2})} 

Minimal penalties for clustering is an active field of research 
---for good reasons since cross-validation cannot directly be applied for clustering, as noticed by Adrien Saumard---, 
and the discussion reflects it. 

Christine Keribin provides a new example where the slope heuristics 
empirically works well for model-based clustering, despite the fact that all models are biased. 
I do not dare suggest as an additional open problem the validation of the slope heuristics in this setting 
---with a significant bias for all models---, 
given the lack of theory available for slope heuristics and model-based clustering. 

Two discussants show the difficulty of developing a well-founded and practical minimal-penalty algorithm for model-based clustering. 
Adrien Saumard points out some mathematical challenges that may explain the lack of theory in this framework. 
I hope his suggests will help solving the problem. 
Jean-Patrick Baudry reports that some extra-parameters of the algorithms used in practice for model-based clustering 
may induce some troubles for the slope heuristics, leading to underfitting. 

Jean-Patrick Baudry's discussion underlines a more general issue: 
when the (theoretical) estimators $\shm(D_n)$ cannot be computed exactly in practice, 
one should not study the estimator-selection problem among $(\shm)_{m \in \M}$ 
but the (more difficult) problem of choosing among the \emph{practical approximations} to these estimators.

\subsection{Large collections of models (Section~\refsurvey{8.3.4})} 
\label{R.sec.openpb.large}

Model selection among a ``large'' collection (using the terminology introduced in Section~\refsurvey{4.7}) 
appears at several places of the discussion. 

For change-point detection, \'Emilie Lebarbier provides numerical experiments 
that support the conjecture of problem (a) in Section~\refsurvey{8.3.4}, with $C^{\star} = \sigma^2$, 
which would be a useful result for reasons detailed in Section~\refsurvey{8.3.4}. 

For orthonormal variable selection by hard thresholding, Sylvain Sardy suggests that 
quantile universal thresholding \citep[QUT;][]{Gia_etal:2017} 
can work well. 
I conjecture that minimal-penalty algorithms 
can also work well in such a setting with a large collection of models, 
because the richness of the collection makes the minimal penalty much larger, 
as shown by partial theoretical results and numerical experiments 
reported in Sections~\refsurvey{4.7} and~\refsurvey{8.3.4}). 

\medbreak 

Nevertheless, a major difficulty for the analysis and the practical use of minimal penalties with large collections of models 
is that the minimal and optimal penalties certainly depend strongly on the noise distribution, 
which should be added to the conjectures formulated in Section~\refsurvey{8.3.4}. 
This is explained well by Lucien Birg\'e, 
who points out the role of large-deviation arguments for the analysis of minimal and optimal penalties in this context. 
David Donoho and Matan Gavish show that a similar phenomenon appears in a slightly different context (factor analysis), 
as detailed above in Section~\ref{R.sec.related.scree}. 
From the practical point of view, the key problem is that minimal-penalty algorithms 
can only be used if the relationship between the minimal and optimal penalty is known 
\emph{before} observing the data. 
In particular, if this relationship depends on the (unknown) noise distribution,  
and if it cannot be easily estimated from the data, 
then minimal-penalty algorithms are useless. 
In such cases, the only hope would be to have a universal relationship between the minimal penalty 
and a suboptimal (but good) penalty, that is, a penalty satisfying an oracle inequality with a bounded leading constant. 
Hopefully, this seems to be possible in the settings respectively mentioned 
by David Donoho and Matan Gavish 
and by Lucien Birg\'e.

\subsection{Infinite estimator collections (Section~\refsurvey{8.3.5})} 

Pierre Bellec shows the noticeable fact that some infinite estimator collections 
---ordered linear smoothers--- can be tuned optimally 
---with an oracle inequality with leading constant \emph{equal to}~$1$ and a remainder term of order $\sigma^2/n$--- 
by an aggregation procedure \citep{Bel_Yan:2019}. 
This phenomenon can be related to the result by 
\citet{Arl_Bac:2009:minikernel_long_v2} on kernel ridge regression, 
that is mentioned in Section~\refsurvey{8.3.5}. 
Let me nevertheless point out three important differences. 
First, the result by \citet{Bel_Yan:2019} makes a milder assumption on the estimator collection: 
kernel ridge regressors form a family of ordered linear smoothers, but other examples exist \citep[see for instance][Section~2]{Kne:1994}. 
Second, the result by \citet{Bel_Yan:2019} does not rely on some explicit discretization of the estimator collection 
---chaining arguments make such a discretization only implicitly---, 
hence it could be easier to generalize further, for instance to multivariate families of kernel ridge regressors. 
Third, the key deviation bound proved by \citet{Bel_Yan:2019} includes a non-negative term 
$c \lVert (A_{\lambda} - A^*) \rVert^2$, 
while generalizing the results of \citet{Arl_Bac:2009:minikernel_long_v2} 
would require a similar uniform deviation bound \emph{without} this term. 

The first two differences make the result of \citet{Bel_Yan:2019} quite interesting for generalizing the analysis of 
minimal-penalty algorithms, 
but the third difference make me join Pierre Bellec's interrogative statement at the end of his discussion. 
At least, this should motivate further work on the topic. 

At a broader level, this underlines that analyzing penalization and aggregation procedures 
require similar concentration inequalities. 
Therefore, solving the open problems mentioned in Section~\refsurvey{8.6} would have 
consequences far beyond minimal-penalty algorithms, 
including the theoretical analysis of aggregation procedures such as Q-aggregation.

\subsection{Phase transitions (Section~\refsurvey{8.5})} 
\label{R.sec.openpb.phase-transitions}

As mentioned in Section~\ref{R.sec.openpb.need-theory}, Lucien Birg\'e suggests to consider the problem of tuning  
$\rho$-estimators thanks to some kind of minimal penalty. 
This open problem might share some common points with the use of phase transitions for tuning 
Goldenshluger-Lepski's and related procedures (Section~\refsurvey{8.5.1}), 
because both are based on pairwise comparisons. 
Note however that the theoretical results mentioned in Section~\refsurvey{8.5.1} are for the least-squares risk, 
while the tuning of $\rho$-estimators should be done for the Hellinger distance, 
which increases the difficulty of the challenge. 

The use of phase transitions for tuning thresholding procedures (Section~\refsurvey{8.5.2})
might also be related to the setting of David Donoho and Matan Gavish's discussion 
(see also the comments made in Sections \ref{R.sec.related.scree} and~\ref{R.sec.openpb.large} above). 
In particular, I notice that the minimal and optimal thresholds depend on the noise distribution  
---through a multiplicative factor $\sigma$ in the Gaussian case---, 
which is a problem if this distribution is not known. 
If the minimal threshold can be estimated from data 
---for instance, with an histogram representation, as suggested by David Donoho and Matan Gavish---, 
it will provide valuable information about the noise distribution, 
which can then be used for computing a (near-)optimal data-driven threshold. 
At least, in the Gaussian case, such a procedure could be used for estimating $\sigma$, 
even if other strategies exist \citep[Section~III.E]{Gav_Don:2014}. 

\section{Overpenalization (Section~\refsurvey{8.4})}

The overpenalization problem appears at least at two places in the discussion. 

Sylvain Sardy shows that the need for overpenalization can occur for tuning the adaptive Lasso, 
because of the high variance of the unbiased estimate of the risk. 
He suggests quantile universal thresholding \citep[QUT;][]{Gia_etal:2017} 
for solving the issue, 
which adds a lead to the resolution of the open problem of 
data-driven overpenalization mentioned in Section~\refsurvey{8.4}. 

Adrien Saumard provides arguments against the fact that minimal-penalty algorithms 
could sharply overpenalize in general. 
This is an important addition to Section~\refsurvey{8.4}, 
in which I may have been too optimistic 
about the power and versatility of minimal-penalty algorithms. 
I am pleased to see that several of us are considering overpenalization as an important issue, 
and I encourage other researchers to join us on this subject 
and explore other ideas to solve it, well beyond minimal penalties.

\bibliography{survey_penmin}

\end{document}